\documentclass{article}
\usepackage[utf8]{inputenc}
\usepackage{amsmath}
\usepackage{graphicx}
\usepackage{multirow}
\usepackage[dvipsnames]{xcolor}
\usepackage{amssymb}
\usepackage{float}
\usepackage{tikz}
\usepackage{amsthm}
\usepackage{mathtools}
\usepackage{hyperref}
\usepackage[parfill]{parskip}
\usepackage{enumerate}
\usepackage{etaremune}
\usepackage[style=trad-abbrv]{biblatex}
\usepackage{comment}
\usepackage{ytableau}
\ytableausetup{aligntableaux = bottom}
\newtheorem{thrm}{Theorem}
\newtheorem{lemma}{Lemma}[section]
\newtheorem{defn}[lemma]{Definition}
\newtheorem{cor}[lemma]{Corollary}
\newtheorem{prop}[lemma]{Proposition}
\newtheorem*{eg}{Example}

\DeclareUnicodeCharacter{0308}{}
\DeclareUnicodeCharacter{0306}{}

\title{On the first order theory of plactic monoids}
\author{Daniel Turaev}
\date{}

\begin{document}

\maketitle
\begin{abstract}
    This paper proves that a plactic monoid of any finite rank will have decidable first order theory. This resolves other open decidability problems about the finite rank plactic monoids, such as the Diophantine problem and identity checking. This is achieved by interpreting a plactic monoid of arbitrary rank in Presburger arithmetic, which is known to have decidable first order theory. We also prove that the interpretation of the plactic monoids into Presburger Arithmetic is in fact a bi-interpretation, hence any two plactic monoids of finite rank are bi-interpretable with one another. The algorithm generating the interpretations is uniform, which  answers positively the decidability of the Diophantine problem for the infinite rank plactic monoid. 
\end{abstract}
\tableofcontents
\section{Introduction}

The plactic monoid has its origin in the work of Knuth \cite{knuth1970permutations}, which based itself on the algorithm developed by Schensted \cite{schensted61}. First studied in depth by Lascoux and Sch\"utzenberger \cite{plaxique81, schutzenberger1977correspondance, schutzenberger1971construction}, its combinatorial properties were applied to the theory of symmetric polynomials to prove the Littlewood-Richardson rule. Due to its origins as a monoid of Young tableaux, it has proved useful in various aspects of geometry and representation theory \cite{fulton_1996}. More recently, it has found application in Kashiwara's crystal basis theory \cite{kashiwaracrstals}, with analogous plactic monoids being defined for different root systems associated to crystal bases \cite{MR2269126, LECOUVEYplcn,LECOUVEYplgn,LECOUVEYplbnpldn,littelmann1995plactic}, and used to study Kostka-Foulkes polynomials \cite{LECOUVEYKostkaFoulkes}. Related, plactic-like monoids have also been defined \cite{stylic, plactic1994gerard, Sylvesterdefn, NOVELLI2000315}, and used to study the combinatorics and growth properties of the plactic monoid, which itself has some interesting combinatorial structure. Cain, Gray, and Malheiro \cite{CainGray15} have shown that the plactic monoids are biautomatic, as are related crystal monoids \cite{caingraycrystal}, and related plactic-like monoids such as the Chinese, Hypoplactic, and Sylvester monoids \cite{Cain_2015}. 

Schensted's multiplication algorithm can be used to decide the word problem for the plactic monoid. It was shown in 1981 that the plactic monoid has decidable conjugacy problem \cite{plaxique81}. A classic generalisation of both the word and conjugacy problems is the Diophantine problem, which has received much attention for free groups \cite{kharlampovich1998irreducible, makanin1982equations, razborov1985systems, SelaGroups}, where Makanin-Razborov diagrams were used independently by Sela \cite{SelaGroups} and Kharlampovich and Myasnikov \cite{kharlampovich1998tarski} to solve the Tarski problems on the first order theory of free groups\footnote{For a survey of these results, see \cite{fine2015tarski}}. The Diophantine problem has also been studied for free monoids \cite{makanin1977problem, SelaMonoids} and is gaining attention in the study of other monoids \cite{GarretaGray2021, CFmonoid}. In the monoid setting, it asks for an algorithm for deciding whether a given system of equations has a solution in a given monoid. 

An active area of research  is the question of checking identities in the plactic monoids and their monoid algebras. Progress has been made in the rank 3 case \cite{kubat2015identities, kubat2012plactic}, and the plactic monoid, bicyclic monoid, and related plactic-like monoids have been shown to admit faithful representations in terms of matrices over the tropical semiring \cite{cain2017note, CAIN2022819, Daviaud_2018, johnson2021tropical}. This implies that every plactic monoid of finite rank satisfies a nontrivial semigroup identity. There is a natural decision problem underpinning this field of study -- is it decidable whether a given identity is satisfied by a plactic monoid. 

In this paper, we show that the plactic monoid of every finite rank has decidable first order theory. This result is a significant generalisation of both of the above ideas. Both identities and Diophantine equations are expressible as first order sentences, thus yielding positive results for both the Diophantine problem and the problem of identity checking. It is not equivalent to these results -- the free semigroup has undecidable theory despite having decidable Diophantine problem \cite{Quine1946-QUICAA}. Nor does this result follow from the multi-homogeneity of the plactic monoids, as there are multi-homogeneous monoids with undecidable theory, and even undecidable conjugacy problem \cite{CainSylvester}. 

The argument presented below is by constructing an interpretation of a plactic monoid in Presburger arithmetic, and could open the door to studying the theories of plactic-like classes of monoid. It is known that all groups interpretable in Presburger arithmetic are abelian-by-finite \cite{ONSHUUS2020102795}. This result may also be a starting point for classifying all monoids interpretable in Presburger arithmetic. 

In preparing this paper for publication, the author was made aware in private communication that Alan Cain and Tara Brough have independently constructed a proof that the plactic monoids are interpretable in Presburger arithmetic, which will appear in a forthcoming paper of theirs. This would overlap with results in sections \ref{P2basecase} and \ref{Pngeneralcase}, but not with any results from sections \ref{sectionBiinterpret} and \ref{sectionInfinite}. To the author's knowledge, the method used to prove this result is different than the one we present.

Sections \ref{P2basecase} and \ref{Pngeneralcase} deal with the construction of an interpretation of a plactic monoid in Presburger arithmetic. In section \ref{sectionBiinterpret} we show that this interpretation is in fact a bi-interpretation, and that certain submonoids of the plactic monoid are definable. In section \ref{sectionInfinite}, we discuss the Diophantine problem in certain infinitely generated monoids generalising the plactic monoid.

\subsection*{Notation and conventions}
Write $A$ for a totally-ordered alphabet on $n$ letters, which will usually be the set $\{1,\dots,n\}$. The free monoid on an alphabet $A$ will be written with a Kleene star $A^*$, and will have identity $\varepsilon$, the empty word. Given a set of generators $A$ and relations $R\subset A^*\times A^*$, the monoid presentation denoted by $\langle A|R\rangle$ will be the quotient of $A^*$ by the congruence generated by $R$. The set $\mathbb{N}$ of natural numbers will contain 0.

\section{Background}
\subsection{Rewriting systems}

A string rewriting system (henceforth rewriting system) for $A^*$ is a set $R~\subset ~A^*\times~A^*$ of elements $(\ell,r)$, usually written $\ell\to r$, called \emph{rewrite rules}. See the book \cite{book1993string} for a more detailed introduction.

For two elements $u,v\in A^*$, write $u\to_R v$ if $u = x\ell z$, $v = xrz$, and $(\ell,r)\in R$. The transitive and reflexive closure of $\to_R$, written $\to_R^*$, is called the reduction relation of $R$. The symmetric closure of $\to_R^*$ is a semigroup congruence. The monoid obtained by taking the quotient of $A^*$ by this congruence is the monoid presented by $\langle A|R\rangle$. Thus every presentation $\langle A|R\rangle$ also corresponds to a rewriting system, which is written  $(A, R)$. 

A rewriting system is called \emph{Noetherian} if it has no infinite descending chain. That is, there is no sequence $u_1,u_2,\dots\ \in A^*$ such that $u_i\to_R u_{i+1}$ for all $i\in\mathbb{N}$.  A rewriting system is called \emph{confluent} if it has the property that, whenever $u\in A^*$ is such that $u\to_R^* u'$ and $u\to_R^* u''$, there exists a $v$ such that $u'\to_R^* v$ and $u''\to_R^*v$. We call a confluent Noetherian rewriting system \emph{complete}. 

Call $u\in (A,R)$ a \emph{reduced word} if there is no subword $\ell$ of $u$ that forms the left hand side of a rewrite rule in $R$. By theorem 1.1.12 of \cite{book1993string}, if $(A,R)$ is a complete rewriting system, then for every $u\in A^*$ there is a \emph{unique, reduced} $v\in A^*$ such that $u\to_R^* v$. This $v$ is called a \emph{normal form} for $u$, and forms a cross-section of the monoid $\langle A|R\rangle$, in the sense that every element of the monoid is equal to exactly one reduced word. We may therefore identify a monoid admitting a complete rewriting system with its set of normal forms, and the multiplication being concatenation followed by reducing to normal form. 
\subsection{The plactic monoid}

We follow the French conventions of Young diagrams having longer rows underneath shorter ones. 

\begin{defn}

 A \emph{Semistandard Young Tableau} (henceforth simply tableau) is a Young diagram with labelled boxes, with labels satisfying the following conditions
    \begin{itemize}
        \item each row weakly increases left to right
        \item each column strongly decreases top to bottom
    \end{itemize}
\end{defn}

\begin{eg} \begin{ytableau}
      3&4\\
      2&3&3\\
    1&1&2&4&4\\
    \end{ytableau} is a tableau. \begin{ytableau}
     4&5\\ 6 & 1\\ 1&2&3
    \end{ytableau} is not a tableau.
\end{eg}

Let $t$ be a tableau with labels taken from $A$. We associate to $t$ a \emph{row reading} in $A^*$. Suppose $t$ is a tableau of $m$ rows, labelled top to bottom as $r_1,\dots, r_m$. The labels of the boxes in each row are an increasing sequence, which can be viewed as a word $r_i\in A^*$. The row reading of $t$ is then $w = r_1r_2\dots r_m \in A^*$. 

We similarly associate a \emph{column reading} to $t$. Denote the columns of $t$ from left to right by $c_1,\dots, c_m$. Each such column corresponds to a strictly decreasing sequence $c_i\in A^*$. The column reading of  $t$ is then $w = c_1\dots c_m \in A^*$. 

\begin{eg}
The tableau $t = $ \begin{ytableau}
        3 \\2 & 3\\1 & 1 & 2 & 2 & 2
    \end{ytableau}  has row reading $$32311222 = 3\ 23\ 11222$$ and column reading $$32131222 = 321\ 31\ 2\ 2\ 2$$
\end{eg}

We now describe Schensted's algorithm. Consider $A = \{1,\dots,n\}$ the totally ordered alphabet, and $w\in A^*$. We may view $w$ as a finite sequence of numbers. Schensted's algorithm is used to study the longest increasing and decreasing subsequences of $w$. The algorithm associates a tableau to $w$ with the property that the number of columns of $w$ is the length of the longest \emph{increasing} sequence, and the number of rows is the length of the longest \emph{strictly decreasing} sequence. See \cite{schensted61} or chapter 5 of \cite{Loth} for more details on this combinatorial structure.

\begin{defn}[Schensted's algorithm]
    We define $P:A^*\to A^*$ to be the map sending a word $w$ to the row reading of a tableau recursively as follows:

Firstly, $P(\varepsilon) = \varepsilon$. Then suppose $w = x_1\dots x_\ell\in A^*$ and $P(x_1\dots x_{\ell-1}) = r_1\dots r_m$, for some rows $r_i$ that form the row reading of a tableau. Then we have:\begin{enumerate}
    \item If $r_mx_\ell$ is a row, then we set $P(r_1\dots r_mx_\ell) = r_1\dots r_mx_\ell$
    \item If not, then we can write $r_m = r_\alpha y r_\beta$, with $y$ being the leftmost letter such that $x_\ell<y$. Such a $y$ must exist, since otherwise $r_mx_\ell$ would be a row. But then $r_\alpha x_\ell r_\beta$ will be a row. So we set $$P(r_1\dots r_mx_\ell) = P(r_1\dots r_{m-1}y) r_\alpha x_\ell r_\beta.$$ 
\end{enumerate}\end{defn}
We call the process in point (2) `bumping the letter $y$'. If $t$ has row reading $r_1\dots r_m$ and column reading $c_1\dots c_k$, then it is straightforward to show that $$P(r_1\dots r_m) = P(c_1\dots c_k) = r_1\dots r_m.$$

Using this algorithm, we may define a monoid of tableaux, which is our object of interest. 

\begin{defn}[The plactic monoid]
    The relation $\sim$ on $A^*$ given by $$u\sim v \iff P(u) = P(v)$$ is a semigroup congruence, and the monoid $A^*/\sim$ with multiplication given by $u\cdot v = P(uv)$ is called the \emph{plactic monoid of rank $n$}. Denote this monoid by $P_n$.
\end{defn}

Knuth \cite{knuth1970permutations} exhibited a set of defining relations $K$ for the plactic monoids of the form $xzy = zxy$ and $yxz=yzx$ for $x<y<z,\ x,y,z \in A$ and $xyx = yxx$ and $yyx = yxy$ for $x<y,\ x,y\in A$. That is, we have 
$$ K = \{xzy = zxy,\ x\leq y <z \}\cup\{ yxz = yzx, \ x<y\leq z\}$$

with $P_n = \langle A|K\rangle$. For each finite rank, it follows that $P_n$ will be finitely presented. Note that the Knuth relations are equivalent to running Schensted's algorithm on all words of length 3.

It was shown by Cain, Gray, and Malheiro in \cite{CainGray15} that the plactic monoid admits a finite complete rewriting system, which we describe here.

We consider two columns $\alpha, \beta$  as words in $A^*$. We say that $\alpha$ and $\beta$ are \emph{compatible}, written $\alpha\succeq \beta$, if $\alpha\beta$ is the column reading of a tableau. Then each pair $\alpha,\beta$ with $\alpha\nsucceq\beta$ yields a rewrite rule. Consider the tableau associated to $P(\alpha\beta)$. Since the number of columns in $P(\alpha\beta)$ is the  length of the longest increasing sequence, and $\alpha,\beta$ are columns, it follows that $P(\alpha\beta)$ will be a tableau with at most two columns.  Therefore this tableau will have column reading $\gamma\delta$, for some columns $\gamma,\delta$ with $\gamma\succeq\delta$, and potentially $\delta = \varepsilon$.

Now consider $\mathcal{C} =\{c_\alpha\ |\ \alpha\in A^*, \alpha\text{ is a column}\}$ to be a set of symbols corresponding to columns in $A^*$. Since $A$ is finite and columns are strictly decreasing sequences, $\mathcal{C}$ is also finite.  Then define $R$ to be the set of all rewrite rules detailed above $$R = \{c_\alpha c_\beta \to c_\gamma c_\delta|\ \alpha,\beta\in A^*\ \alpha\nsucceq\beta\}$$

It is shown in \cite{CainGray15} that
\begin{lemma}\label{rewritelemma}
   $(\mathcal{C}, R)$ is a complete rewriting system for $P_n$. 
\end{lemma}

It follows from this that $P_n$ admits normal forms as reduced words in $\mathcal{C}^*$. By the definition of $\succeq$, this normal form will be in the form of column readings $c_{\alpha_1}\dots c_{\alpha_m}$ with each $\alpha_i\succeq \alpha_{i+1}$. 

Note that if $\alpha = \alpha_m\dots \alpha_1$ and $\beta = \beta_n\dots \beta_1$, $\alpha_i,\beta_i\in A$, are columns appearing in the column reading of the same tableau (not necessarily adjacent) with $\alpha$ further left than $\beta$,  then $\alpha\succeq\beta$. Indeed, since $\alpha$ and $\beta$ are columns of the same tableau, then by the structure of a tableau we have that $m\geq n$. Furthermore, each pair $\alpha_i,\beta_i$ will be in the same row of the tableau, with $\alpha_i$ appearing earlier than $\beta_i$. This will imply that $\alpha_i\leq \beta_i$. But these two conditions imply that $\alpha\succeq \beta$. Thus $\succeq$ is a partial order on $\mathcal{C}$.

 We introduce  a length-decreasing-lexicographic order on $\mathcal{C}$ extending $\succeq$. For $c_\alpha, c_\beta \in \mathcal{C}$, define: \begin{equation*}
   c_\alpha\sqsubseteq c_\beta \iff \left(|\alpha|>|\beta|\right) \lor\left( |\alpha|=|\beta|\wedge \left(\exists j: \ i<j\implies \alpha_i = \beta_i \wedge \ \alpha_j<\beta_j \right) \right)
\end{equation*}
With $j$ taken as $n+1$ when $c_\alpha = c_\beta$. Note that $c_\alpha\succeq c_\beta \implies c_\alpha\sqsubseteq c_\beta$. Furthermore, this is clearly a total order. We can therefore enumerate the set $\mathcal{C}$ as $\{c_1,\dots c_k\}$, with $k = |\mathcal{C}| = 2^n-1$, such that $i\leq j\implies c_i\sqsubseteq c_j$. Then, since $c_\alpha\succeq c_\beta \implies c_\alpha\sqsubseteq c_\beta$, we have that the normal forms of $P_n$ will have the form $$c_1^{w_1}\dots c_k^{w_k}$$

with $w_i\in\mathbb{N}$ for each $i$, and for any pair $c_i, c_j$ with $i<j \land c_i\nsucceq c_j$, either $w_i = 0 $ or $w_j = 0$. Call two columns $c_i$ and $c_j$ \emph{incompatible} if $i<j \land c_i\nsucceq c_j$.

\begin{eg}
    $P_3$ has seven columns:

    \begin{ytableau}
        3\\2\\1
    \end{ytableau}, \begin{ytableau}
        2\\1
    \end{ytableau}, \begin{ytableau}
        3\\1
    \end{ytableau}, \begin{ytableau}
        3\\2
    \end{ytableau}, \begin{ytableau}
        1
    \end{ytableau}, \begin{ytableau}
        2
    \end{ytableau},  \begin{ytableau}
        3
    \end{ytableau}

listed here in length-decreasing-lexicographic order. This list corresponds to symbols $c_1,\dots, c_7 \in \mathcal{C}$. Note that $c_4$ and $c_5$ are incompatible. $P_3$ is the lowest rank plactic monoid with an incompatible pair.
\end{eg}

\begin{eg}
    In $P_3$, the element $c_1^3c_2c_4^2\in\mathcal{C}^*$ is in normal form and corresponds to the following tableau: 
    
    \ytableaushort{333,222233,111122}
\end{eg}

\subsection{Interpretations, theories, and Presburger arithmetic}

We assume familiarity with basic first order logic, and refer the reader to \cite{hodgesMT} or \cite{marker2013MT} for a more detailed introduction to model theory. We will be following the conventions of \cite{marker2013MT}

\begin{defn}
    Let $L$ be the language of first order formulas for a given signature. Then the first order theory of an $L$-structure $\mathcal{M}$ is the set of all sentences in $L$ that hold in $\mathcal{M}$.

    The question of deciding a first order theory asks for an algorithm which, given a first order sentence $\phi$,  determines whether $\phi$ is true or false in $\mathcal{M}$ in finite time. 
\end{defn}

The language of interest in this paper is the language of monoids, whose signature is $(\circ,\varepsilon)$. 
To speak of the first order theory of a given monoid, one classically allows atomic formulas of the form $u=v$ for each $u,v\in\mathcal{M}$. In the finitely generated case (with generating set $A = \{a_1,\dots,a_n\}$, say) this is equivalent to adding constants $a_1,\dots,a_n$ to the signature, and considering the first order theory with constants of $(\mathcal{M},\circ,\varepsilon,a_1,\dots,a_n)$.

In our case, we refer to the first order theory of a plactic monoid of rank $n$, which will have constants $1,\dots,n$ added to the language of monoids. Write $FOTh(P_n)$ as shorthand for the first order theory of $P_n$ with constants.

We aim in the following sections to build an interpretation of $P_n$ in Presburger arithmetic, which will allow $\varphi\in FOTh(P_n)$ to be reduced to a sentence $\Tilde{\varphi}$ of Presburger arithmetic. A \emph{reduction} of a decision problem $D_1$ to another decision problem $D_2$ is a Turing machine which, given finitely many queries to an oracle for $D_2$, will yield an algorithm for deciding $D_1$. Importantly, this means that decidability of $D_2$ will imply decidability of $D_1$, as such an oracle machine will exist and halt in finite time on each query.

Presburger arithmetic is named after Mojżesz Presburger, who in 1929 was tasked with studying the decidability of the integers under addition. In his master's thesis \cite{Presburger}, he used quantifier elimination and reasoning about arithmetic congruences to prove that the first order theory of $(\mathbb{N},0,1,+)$ is consistent, complete, and decidable. Note that we can add a comparison symbol $\leq$ to the signature of Presburger arithmetic without trouble, since $x\leq y$ is equivalent to the statement $\exists z: y = x+z$. This yields the following lemma:

\begin{lemma}\label{Presburger}
    $FOTh(\mathbb{N},0,1,+,\leq)$ is decidable.
\end{lemma}
This result will form the bedrock of the following argument. For an English translation of Presburger's work, see \cite{MR1111343} or \cite{PresburgerEng}.

For the definition of an interpretation, we proceed as in section 1.3 of \cite{marker2013MT}

\begin{defn}[Definable sets]
    Let $\mathcal{M}$ be an $L$-structure. A set $S\subseteq \mathcal{M}^n$ is called \emph{definable in $\mathcal{M}$} if there is a first order formula $\phi(x_1,\dots,x_n,y_1,\dots,y_m)\in L$ with free variables $x_1,\dots,x_n,y_1\dots,y_m$ such that there exists $(w_1,\dots,w_m)\in \mathcal{M}^m$ with the property that $\phi(x_1,\dots,x_n,w_1,\dots,w_m)$ holds if and only if $(x_1,\dots,x_n)\in S$. i.e. $S$ is the set $$\{\underline{x}\in\mathcal{M}^n|\ \mathcal{M}\models \phi(\underline{x},w_1,\dots,w_m)\}$$
\end{defn}

\begin{eg}
The set of elements that commutes with a given $m\in\mathcal{M}$ is definable by the formula $xm=mx$.  The centre of $\mathcal{M}$ is definable by the formula $$\forall m: xm = mx$$ If $\mathcal{M}$ is finitely generated by $\{m_1,\dots m_k\}$, then the formula $$xm_1=m_1x\land xm_2 = m_2x\land\dots\land xm_k=m_kx$$
defines the centre of $\mathcal{M}$. This has the property of being a \emph{positive existential} formula, which is useful in the study of Diophantine equations\footnote{See, for example, \cite{ciobanu2022group}} 
\end{eg}
\begin{defn}[Definable functions]
    A function $f:\mathcal{M}^m\to\mathcal{M}^n$ is definable in $\mathcal{M}$ if its graph is definable as a subset of $\mathcal{M}^{m+n}$.
\end{defn}
Note that the composition of definable functions is definable. 

\begin{defn}[Interpretability]\label{interpretdefn}
Let $\mathcal{M}$ be an $L_1$-structure, and $\mathcal{N}$ be an $L_2$-structure. Then we call $\mathcal{N}$ \emph{interpretable} in $\mathcal{M}$ if there exist some $n\in\mathbb{N}$, some set $S\subseteq \mathcal{M}^n$, and a bijection $\phi:S \to \mathcal{N}$ such that\begin{enumerate}
    \item $S$ is definable in $\mathcal{M}$
    \item For every $r$ in the signature of $L_2$, including the equality relation, the preimage by $\phi$ of the graph of $r$ is definable in $\mathcal{M}$
\end{enumerate}  
\end{defn}

We will use the notation $\phi^{-1}(r)$ for the preimage of the graph of $r$. Since we will only be dealing with the case of a monoid, point 2 reduces to checking the preimages of equality $\phi^{-1}(=)$ and multiplication $\phi^{-1}(\cdot)$, with the latter being the set of triples $(a,b,c)\in S^3$ such that $\phi(a)\cdot\phi(b) = \phi(c)$. 

Note that in the above definition we insisted the map $\phi$ be a bijection, as in  section 1.3 of \cite{marker2013MT}. The interpretation we will build will be a bijection. However, the most general theory of interpretations works with surjections from $S$ onto $\mathcal{N}$. See section 5 of \cite{hodgesMT} for more information.

The following result will prove fundamental, and is a consequence of theorem 5.3.2 and its remarks in \cite{hodgesMT}:
\begin{prop}\label{reductionlemma}
    Suppose $L_1$ and $L_2$ are languages, with $M_1$ and $M_2$ being $L_1$- and $L_2$-structures, respectively. Suppose $M_1$ is interpretable in $M_2$. Then the problem of deciding $FOTh(M_1)$ is reducible to the  problem of deciding $FOTh(M_2)$.
\end{prop}

Next we define the notion of bi-interpretability, which will be the subject of section \ref{sectionBiinterpret}. By the definition of interpretations, it is straightforward to see that interpretations are transitive: if $M_1$ is interpretable in $M_2$, and $M_2$ is interpretable in $M_3$, then $M_1$ is interpretable in $M_3$. This implies that if two structures are \emph{mutually interpretable}, i.e. $M_1$ and $M_2$ are each interpretable in the other, then we obtain an interpretation of $M_1$ in itself, and likewise an interpretation of $M_2$ in itself.
\begin{defn}[Bi-interpretability]
    Given $M_1$ an $L_1$-structure, and $M_2$ an $L_2$-structure, we say $M_1$ and $M_2$ are \emph{bi-interpretable} if $M_1$ and $M_2$ mutually interpretable, and the map $\phi_i$ interpreting $M_i$ in itself is definable in $M_i$, for $i= 1,2$. 
\end{defn}

\begin{eg}
    Presburger arithmetic is commonly expressed as $(\mathbb{N},0,1,+,\leq)$ and $(\mathbb{Z},0,1,+,-,\leq)$. These two models are bi-interpretable. 
\end{eg}    
    Indeed, the identity map $\phi$ on $\mathbb{N}\subset\mathbb{Z}$ definable by $0\leq x$ interprets $(\mathbb{N},0,1,+,\leq)$ in $(\mathbb{Z},0,1,+,-,\leq)$, and the map $\psi: \mathbb{N}^2 \to \mathbb{Z}$ with $\psi(a,b) = a-b$ interprets as a surjection $(\mathbb{Z},0,1,+,-,\leq)$ in $(\mathbb{N},0,1,+,\leq)$. We then obtain $\phi\psi: S \to \mathbb{N}$ from the definable set $S = \{(a,b)\ |\ a\geq b\}\subset \mathbb{N}^2$, with graph $$\{(a,b,c)\ |\ a\geq b \land c+b = a\}$$ definable in $(\mathbb{N},0,1,+,\leq)$. We also obtain $\psi\phi^2: T\to \mathbb{Z}$ with $T = \mathbb{N}^2$ definable in $\mathbb{Z}^2$ by $0\leq x\land 0\leq y$, whose graph $$\{(a,b,c)|\ 0\leq a\land 0 \leq b \land c = a-b\}$$ 
definable in $(\mathbb{Z},0,1,+,-,\leq)$.

We will use these two notions of Presburger arithmetic interchangeably.

\section{The case $n=2$}\label{P2basecase}
First, we will explicitly treat the case $n=2$ of tableaux on two letters.  Such tableaux have three possible columns:

\begin{ytableau}
    2\\1
\end{ytableau},$\ $ \begin{ytableau}
    1
\end{ytableau},$\ $ \begin{ytableau}
    2
\end{ytableau}.

 Denote by $t$ the word $21$ in $A^*$. Then by abuse of notation $\mathcal{C}^* = \{t,1,2\}^*$, and our rewriting system becomes: $$ R = \{21 \to t\ ,\ 2t \to t2\ ,\ 1t\to t1\}$$
By the two commutativity rules, and the fact that any factor $21$ would not appear in a reduced word, we can write any reduced word $w\in (\mathcal{C},R)$ as some $t^{w_1}1^{w_2}2^{w_3}$. Thus, by completeness of the rewriting system, each element of $P_2$ corresponds to a triple $(w_1,w_2,w_3)\in\mathbb{N}^3$, associated to a normal form $t^{w_1}1^{w_2}2^{w_3}$. Likewise, each such triple corresponds to a tableau, hence an element of $P_2$. 

We will take the $\mathbb{Z}$ version of Presburger arithmetic, to make use of subtraction. Consider the map $\phi:S \to P_2$, where $S = \mathbb{N}^3\subset \mathbb{Z}^3$ is definable by the formula $$(0\leq x_1)\land (0\leq x_2)\land (0\leq x_3)$$
and $\phi(x_1,x_2,x_3) = t^{x_1}1^{x_2}2^{x_3}$. This is a bijection from a definable set in  $(\mathbb{Z},0,1,+,-,\leq)$, and the inverse graph of equality will be 
\begin{align*}
    \phi^{-1}(=) &= \left\{ (a_1,a_2,a_3,b_1,b_2,b_3) \in \mathbb{N}^6\ |\ t^{a_1}1^{a_2}2^{a_3} = t^{b_1}1^{b_2}2^{b_3}\right\}\\
    &=\left\{ (a_1,a_2,a_3,b_1,b_2,b_3) \in \mathbb{N}^6\ |\ a_1 = b_1,\ a_2 = b_2,\ a_3=b_3\right\}  \subset \mathbb{Z}^6
\end{align*}
which is definable by the formula $(\underline{a}\in S)\land(\underline{b}\in S)\land \bigwedge\limits_{i \in \{1, 2, 3\}} (a_i~=~b_i)$.

Next we check the preimage of the graph of multiplication
\begin{align*}
    \phi^{-1}(\circ) = \left\{(\underline{a},\underline{b},\underline{c}) \in S^3\ |\ t^{a_1}1^{a_2}2^{a_3}t^{b_1}1^{b_2}2^{b_3} = t^{c_1}1^{c_2}2^{c_3}\right\}.
\end{align*}
Explicitly checking the multiplication yields
\begin{align*}
    t^{a_1}1^{a_2}2^{a_3}t^{b_1}1^{b_2}2^{b_3} &= t^{a_1+b_1}1^{a_2}2^{a_3}1^{b_2}2^{b_3}\\
    &=\begin{cases}
        t^{a_1+b_1 + a_3}1^{a_2+b_2-a_3}2^{b_3},\ a_3\leq b_2\\
        t^{a_1+b_1 + b_2}1^{a_2}2^{b_3 + a_3 -b_2},\ b_2\leq a_3
    \end{cases}
\end{align*}
 thus obtaining the following formula for $\phi^{-1}(\circ)$
\begin{align*}
    (\underline{a}\in S)\land(\underline{b}\in S)\land(\underline{c}\in S)\land&[(a_3\leq b_2\land c_1 = a_1+b_1+a_3\land c_2 = a_2+b_2-a_3\land c_3 = b_3)\\
    \lor&(b_2\leq a_3 \land c_1 = a_1+b_1+b_2 \land c_2 = a_2 \land c_3 = b_3 + a_3 - b_2)]
\end{align*}

where $\underline{a}\in S$ is a shorthand for$(a\leq x_1)\land (0\leq a_2)\land (0\leq a_3)$. It follows then that $\phi$ is an interpretation of $P_2$ in Presburger arithmetic. This yields the following result.

\begin{thrm}
    $P_2$ has decidable first order theory
\end{thrm}
\begin{proof}
Since $\phi$ above is an interpretation of $P_2$ in  $(\mathbb{Z},0,1,+,-,\leq)$, every first order formula of $P_2$ is interpreted as a first order formula of Presburger arithmetic, which is decidable by lemma \ref{Presburger}.
\end{proof}
Note that this argument is closely related to the proof that the bicyclic monoid $B = \langle a,b\ |\ ba = \varepsilon\rangle$ has decidable first order theory (see section 2.4 of \cite{diekert2008}). Indeed, the map $\psi:P_2 \to B$ sending 1 to $a$, 2 to $b$, and $t$ to $\varepsilon$ is a monoid homomorphism, and $\psi\circ\phi:S \to B$ is an interpretation of the bicyclic monoid in Presburger arithmetic, in the sense of surjections.

\section{The general case}\label{Pngeneralcase}
 Throughout this section, let $k = |\mathcal{C}| = 2^n-1$. Index $\mathcal{C}$ by $i \in \{1,\dots, k\}$ with $i~<~j~\iff~c_i~\sqsubset~c_j.$

Let $S\subseteq \mathbb{N}^k$ be the set of all $(v_1,\dots,v_k)$ such that $c_1^{v_1}\dots c_k^{v_k}$ is the normal form of a tableau, and let $\phi:S\to P_n$ be the natural bijection. The normal form of any tableau will obey compatibility conditions: for each pair $$(a,b)\in \{1,\dots,k\}\times \{1,\dots,k\}$$ such that $a<b$ and $c_a\nsucceq c_b$, we have that either $v_a= 0$ or $v_b = 0$. Let $I~\subset~\{1,\dots,k\}^2$ be the set of all such pairs. Then $S\subset\mathbb{Z}^k$ is defined by the formula$$\bigwedge\limits_{i\in\{1,\dots,k\}} (0\leq x_i) \land \bigwedge\limits_{(a,b) \in I} \left[(x_a = 0)\lor(x_b = 0)\right].$$

We claim that $\phi$ is an interpreting map of $P_n$ in Presburger arithmetic. Again, we check the diagonal:
\begin{align*}
    \phi^{-1}(=) = \{(\underline{a},\underline{b}) \in S^2\ |\ \phi(\underline{a}) = \phi(\underline{b})\},
\end{align*}

which is definable by $(\underline{a}\in S)\land(\underline{b}\in S)\land \bigwedge\limits_{i\in\{1,\dots,k\}}(a_i = b_i)$ as in the $n=2$ case. It remains to check whether the preimage of the multiplication graph $$ \phi^{-1}(\circ) = \left\{(\underline{a},\underline{b},\underline{c}) \in S^3\ |\ \phi(\underline{a})\phi(\underline{b}) = \phi(\underline{c})\right\}\subset \mathbb{Z}^{3k}$$
is definable.

\subsection{Multiplication -- the idea}
Using section \ref{P2basecase} as a base case, we will proceed with the induction hypothesis that, for each $2\leq i \leq n-1$, we have a formula $\eta_i$ in Presburger arithmetic defining multiplication in $P_i$.

We first consider the structure of multiplication in $P_n$. The recursive nature of Schensted's algorithm yields a characterisation of multiplication in $P_n$ via bottom rows and top tableaux.

\begin{defn} We call a tableau $t\in P_n$ a \emph{top tableau} if its row reading is a word over $\{2,\dots, n\}^*$. i.e. there are no 1's appearing in the tableau word representing $t$.
\end{defn}

Note that each $u\in P_n$ will have an associated top tableau: if $u = r_1\dots r_l$, where each $r_i$ is a row, then $r_1\dots r_{l-1}$ will be a top tableau.

For $u,v\in P_n$, the product $uv$ will be computed by first running an insertion algorithm into the bottom row of $u$, and then inserting any bumped letters into the top tableau associated to $u$. We will make this idea more precise.

\begin{defn}\label{def:topbot} Define the following maps:
    \begin{enumerate}
           \item The top map $T:P_n \to P_n$ maps an element $w$ with row form $r_1\dots r_l$ in $A^*$ to its corresponding top tableau $T(w) = r_1\dots r_{l-1}$
        \item The bottom map $B:P_n \to P_n$ maps an element $w$ as above to its bottom row $B(w) = r_l$
    \end{enumerate}
\end{defn}
\begin{eg}
    Let $t = $ \begin{ytableau}
      3&4\\
      2&3&3\\
    1&1&2&4&4\\
    \end{ytableau}. Then $T(t) = 34\ 233$ and $B(t) = 11244$
\end{eg}

For $u,v\in P_n$, by the structure of Schensted's algorithm, the product $uv$ will run an insertion algorithm first into $B(u)$, followed by any letters that are bumped being inserted into $T(u)$. This yields the following characterisation of the top and bottom of the product:
    \begin{align*}
        T(uv) &= T(u)T(B(u)v)\\
        B(uv) &= B(B(u)v)
    \end{align*}
Where equality is taken to mean equality in $P_n$, not equality of words. Note that the set of top tableaux, which is equivalently the image of $T$, is a submonoid isomorphic to $P_{n-1}$ over the alphabet $\{2,\dots,n\}$. Thus by our induction hypothesis the product of top tableaux will be definable via $\eta_{i-1}$. Therefore, if we can define the row $B(uv)$, and a way of stitching $T(uv)$ and $B(uv)$ into one tableau $uv$, we will obtain $\eta_i$ a formula defining multiplication in $P_n$
\begin{defn}
      The stitch map $\Sigma:P_n\times P_n\to P_n$ is defined as follows. For $u\in P_n$ a top tableau with row reading $r_1\dots r_n\in A^*$ and $v$ a row with row reading $r_v\in A^*$,  $\Sigma(u,v) = uv$ if $r_1\dots r_nr_v$ is the row reading of a tableau. Otherwise, $\Sigma(u,v) = \varepsilon$.       
\end{defn}
If $\Sigma$ has nontrivial output, we call $u$ and $v$ \emph{compatible}, and $uv$ the ``stitched" tableau. 
\begin{eg}
    Suppose $u = 43322234$ and $v = 11113$. Then $\Sigma(u,v) = uv$, with corresponding tableau 
    
    \centering{\begin{ytableau}
        4\\3&3\\2&2&2&3&4\\1&1&1&1&3
    \end{ytableau}}.
\end{eg}

Note that $\Sigma(T(u),B(u)) = u$. We can thus characterise multiplication via the above maps as follows \begin{align*}
    uv &= \Sigma(T(uv),B(uv))\\
     &=\Sigma(T(u)T(B(u)v),B(B(u)v)).
\end{align*}
Let us consider the structure of $w\in P_n$ as a word in normal form in $\mathcal{C}^*$. We have that $w = c_1^{w_1}c_2^{w_2}\dots c_k^{w_k}$ for some $w_i\in\mathbb{N}$ for each $i$, satisfying some compatibility conditions. But consider now each block $c_i^m$ for some $m\in\mathbb{N}$ as a tableau word in row form in the presentation $\langle A|K\rangle$. Then for $c_i = x_1x_2\dots x_r$ in row form in $A^*$, we have that $c_i^m = x_1^mx_2^m\dots x_r^m$ in row form in $A^*$. For each $c_i$, this row form is unique, since each column corresponds to a unique decreasing sequence in $A^*$. This will give us a useful way to write the column form of $w$ as a word in $A^*$.

Define $\alpha = (\alpha_1,\dots,\alpha_\ell)$ to be the finite sequence of letters in $A$ which first outputs in order the letters in the row form of $c_1$, then the letters of the row form of $c_2$, and so on. We also define $\beta = (\beta_1,\dots,\beta_\ell)$ to be the finite sequence, taking values in $\{1,\dots,k\}$, with $\beta_i = j$ when $\alpha_i$ is a letter from column $c_j$. Note that these sequences only depend on the rank of $P_n$, as it is defined using the columns generating $P_n$. Now, a straightforward check using Schensted's algorithm verifies that the word $ \alpha_1^{w_{\beta_1}}\dots \alpha_\ell^{w_{\beta_\ell}} \in A^*$ is equal to $w = c_1^{w_1}c_2^{w_2}\dots c_k^{w_k}$ in the plactic monoid, where $w_{\beta_i}$ is the coefficient of column $c_{\beta_i}$ in the normal form of $w$.  

\begin{eg}
    The seven columns of $P_3$:
    
    \begin{ytableau}
        3\\2\\1
    \end{ytableau}, \begin{ytableau}
        2\\1
    \end{ytableau}, \begin{ytableau}
        3\\1
    \end{ytableau}, \begin{ytableau}
        3\\2
    \end{ytableau}, \begin{ytableau}
        1
    \end{ytableau}, \begin{ytableau}
        2
    \end{ytableau},  \begin{ytableau}
        3
    \end{ytableau}, yield \begin{align*}
        \alpha &= 3,2,1,2,1,3,1,3,2,1,2,3\\
        \beta &= 1,1,1,2,2,3,3,4,4,5,6,7.
    \end{align*} 
    For an example word $w = c_1^3c_2c_3c_6$, we get the corresponding sequence $$3^32^31^32^11^13^11^13^02^01^02^13^0 = 33322211121312,$$ which after running Schensted's algorithm becomes the tableau 
    
    \centering{\ytableaushort{333,22223,111112}.}
\end{eg}

\begin{lemma}\label{sequencemultn}
    Consider $u,v,w \in P_n$. Suppose $v$ has normal form $c_1^{v_1}\dots c_k^{v_k}$. Then $w = uv$ is equivalent to the following:

    There exist $u_0,u_1,\dots ,u_\ell \in P_n$ such that $u_0 = u$, $u_\ell = w$, and we have a recursive formula for $u_i$
    $$u_i = u_{i-1}\alpha_i^{v_{\beta_i}}$$
\end{lemma}
This result is immediate from the structure of the insertion algorithm and the fact $v = \alpha_1^{v_{\beta_1}}\dots \alpha_\ell^{v_{\beta_\ell}}$

\begin{defn}
    For each $x\in A$ the map $\mu_x:\mathbb{N}\times S\to S$  is such that $\phi(\mu_x(m, \underline{a})) = \phi(\underline{a})x^m$
\end{defn}

\begin{cor}\label{simplification}
$\phi^{-1}(\circ)$ is definable if the maps $\mu_x$ are definable for each $x\in A$. 
\end{cor}
\begin{proof}
    By lemma \ref{sequencemultn}, given $\underline{a},\underline{b}\in S$, we have that $\phi(\underline{c}) = \phi(\underline{a})\phi(\underline{b})$ if and only if there is some $\underline{c}^0,\dots, \underline{c}^\ell$ such that $\underline{c}^0 = \underline{a},\ \underline{c} = \underline{c}^\ell$, and $$\underline{c}^i = \mu_{\alpha_i}(b_{\beta_i},\underline{c}^{i-1})$$
    so the preimage of the graph of multiplication is a composition of finitely many applications of $\mu_x$, which will be definable if each $\mu_x$ is definable.
\end{proof}

\subsection{The formula defining $\mu_x$}
Henceforth, $x$ is a fixed letter in $A$.

Recall that if $\underline{b} =\mu_x(m,\underline{a})$, then \begin{align*}
    \phi(\underline{b}) &= \phi(\underline{a})x^m\\ &=\Sigma(T(\phi(\underline{a}))T(B(\phi(\underline{a}))x^m),B(B(\phi(\underline{a}))x^m)).
\end{align*}
So we wish to obtain a formula of Presburger arithmetic describing $$\underline{b} = \phi^{-1}\Sigma(T(\phi(\underline{a}))T(B(\phi(\underline{a}))x^m),B(B(\phi(\underline{a}))x^m))$$
We can break this down into a composition of several maps. First, define $\underline{a}^1$ and $\underline{a}^2$ such that 
\begin{align*}
    \underline{a}^1 &= \phi^{-1}T\phi(\underline{a})\\
    \underline{a}^2 &= \phi^{-1}B\phi(\underline{a})
\end{align*}
Next, considering $R\subset P_n$ the set of row words, we define two maps $\rho_1,\rho_2: R~\to~S$ such that, for $r\in R$, $\phi(\rho_1(r)) = T(rx^m)$ and $\phi(\rho_2(r)) = B(rx^m)$. Then since $\phi(\underline{a}^2)$ is a row, we can define $\underline{a}^3$ and $\underline{a}^4$ to be such that 
\begin{align*}
    \underline{a}^3 = \rho_1(\phi(\underline{a}^2))\\
    \underline{a}^4 = \rho_2(\phi(\underline{a}^2))
\end{align*}
That is, $\phi(\underline{a}^3) = T(B(\phi(\underline{a}))x^m)$ and $\phi(\underline{a}^4) = B(B(\phi(\underline{a}))x^m)$. 

Next, define $\underline{a}^5$ to be such that $$\phi(\underline{a}^5) = T(\phi(\underline{a}))T(B(\phi(\underline{a}))x^m) = \phi(\underline{a}^1)\phi(\underline{a}^3).$$
By our induction hypothesis, this will be definable in Presburger arithmetic, as the coefficients in $\underline{a}^5$ will either be calculated by the formula $\eta_{n-1}$, or will equal zero. 

Finally, we have that $$\underline{b} = \phi^{-1}\Sigma(\phi(\underline{a}^5),\phi(\underline{a}^4)).$$ Since the composition of definable maps is definable, we have that $\mu_x$ is definable precisely when $\phi^{-1}T\phi(\underline{a}), \phi^{-1}B\phi(\underline{a}), \rho_1, \rho_2,$ and $\phi^{-1}\Sigma(\phi(\ ),\phi(\ ))$ are definable. This will be the subject of the following three lemmas.
\begin{lemma}\label{TBPEdefd} The following maps are definable:
\begin{enumerate} [(i)]
\item $\phi^{-1}B\phi: S \to S$ 
\item $\phi^{-1}T\phi: S \to S$ 
\end{enumerate}
\end{lemma}
\begin{proof} 
$(i)$ Define the finite sets $B_a$ for each $a\in A$ by $$B_a = \{j\in\{1,\dots,k\}\ |\ c_j = x_m\dots x_1 \land x_1 = a\}$$ which are nonempty for each $a$. Then we get that $\underline{b} = \phi^{-1} (B(\phi(\underline{a}))$ if and only if the following formula holds: $$\bigwedge\limits_{i\in \{1,\dots,k-n\}} (b_i = 0) \land \bigwedge\limits_{i\in\{1,\dots,n\}} \left(b_{k-n+i} = \sum_{j\in B_i} a_j\right)$$
The first part of the formula denoting the coefficient of each column of size $\geq2$ being zero, and the second part denoting the fact that each column $x_m \dots x_1$ in $\phi(\underline{a})$ contributes to the coefficient of the $x_1$ letter in the bottom row. 

$(ii)$ Define the similar sets $T_i$ for each $i \in \{1,\dots,k\}$ by $$T_i = \{j \in \{1,\dots,k\}\ |\ c_j = x_m\dots x_1 \land x_m\dots x_2 = c_i\}$$

Note that if $i \in B_1$, then $T_i = \emptyset$. Now, we have that $\underline{b} = \phi^{-1}(T(\underline{a}))$ if and only if the following formula holds: $$\bigwedge\limits_{i\in\{1,\dots,k\}} \left(b_i = \sum_{j \in T_i} a_j\right)$$
Where we take the sum over an empty indexing set to be 0. 
\end{proof}

Note that the sets $T_i$ and $B_a$ can be constructed algorithmically for any given  $n$. Given the set of columns as decreasing sequences, we can check membership in each $B_a$ by considering the minimal element of a column, and we can check membership in each $T_i$ by considering the column without its minimal element. Note also that $\{B_a: a\in A\}$ and $\{T_j: j\in \{1,\dots,k\},\ T_j\neq \emptyset\}$ are partitions of $\{1,\dots,k\}$.

Next, we move on to defining the maps  $\phi^{-1}T(rx^m)$ and $\phi^{-1}B(rx^m)$. Here, we consider $S_R \subset S$ to be the subset of normal forms corresponding to rows (i.e. $S_R$ is the preimage of $R\subset P_n$).
\begin{lemma}\label{rowmultn} The following maps are definable:
\begin{enumerate}
    \item $\rho_1\phi: S_R \to S$ taking $\phi^{-1}(r)$ to $\phi^{-1}T(rx^m)$
    \item $\rho_2\phi: S_R \to S$ taking $\phi^{-1}(r)$ to $\phi^{-1}B(rx^m)$
\end{enumerate}
\end{lemma}
\begin{proof} First note that $S_R$ is a subset definable by $$\bigwedge_{i\in\{1,\dots,k-n\}} x_i = 0$$
and $\rho_1\phi, \ \rho_2\phi$ will be maps from $S_R$ to $S_R$. Indeed, for $\rho_2\phi$ this is immediate, but for $\rho_1\phi$ note that by \cite{schensted61} the number of rows after running Schensted's algorithm on any $w \in A^*$ is equal to the length of the longest strictly decreasing subsequence of $w$. Now, since $r$ is non-decreasing as a sequence in $A^*$, the longest strictly decreasing subsequence of $rx^m$ viewed as a word in $A^*$ can have length at most 2. Thus the top $T(rx^m)$ can have at most one row, meaning the image of $\rho_1\phi$ must be in $S_R$. 

    Now, write $\underline{r} = (0,\dots,0,r_1,r_2,\dots,r_n)$. We will describe explicitly $\underline{c} = (0,\dots,0,c_1,\dots,c_n)$ and $\underline{d}~=~(0,\dots,d_1,\dots,d_n)$ such that $\underline{c} = \rho_1\phi(w)$ and $\underline{d} = \rho_2\phi(w)$.

    \begin{etaremune}
        \item We will first consider the $\rho_2\phi$ case, which corresponds to $B(w)$. In the setting of the presentation $\langle A|K\rangle$, we will have $x^m$ inserted into $$r = 1^{r_1}2^{r_2}\dots(x+1)^{r_{x+1}}(x+2)^{r_{x+2}}\dots n ^{r_n}$$ It will bump $m$ letters from this row, starting at $x+1$. This means that $d_i = r_i$ for $i<x$ and $d_x = r_x+m$. We will now consider the later entries of $\underline{d}$, which will split into several cases depending on the size of $m$.

    In the first case, suppose $m\leq r_{x+1}$. Then we will bump $m$ letters $x+1$ and replace them with letters $x$. This yields the effect that $d_{x+1} = r_{x+1} - m$ and $d_i = r_i$ for all $i>x+1$.

    In the next case, suppose $r_{x+1}\leq m \leq r_{x+1}+r_{x+2}$. Then all letters $x+1$ are bumped, as are $m-r_{x+1}$ letters $x+2$. Thus we have that $d_{x+1} = 0,\ d_{x+2} = r_{x+2} + r_{x+1} - m$, and $d_i = r_i$ for all $i>x+2$. 

    Generalising the above, suppose, for some $i\leq n-x$, we have \begin{align*}
        \sum_{j=1}^{i-1} r_{x+j}\leq m \leq \sum_{j=1}^{i} r_{x+j}
    \end{align*}
    Then in this case, $d_{x+j} = 0$ for each $0<j<i$, and $d_{x+i} = \sum_{j=1}^{i} r_{x+j} - m$, and all later entries remain unchanged. 

    The last case to consider is when  \begin{align*}
        \sum_{j=1}^{n-x} r_{x+j}\leq m
    \end{align*}
    In which case all letters bigger than $x$ will be bumped and we have $d_{x+j}~=~0$ for all $j$. 

    Each case yields a formula in terms of $\leq$, addition, and subtraction. Then the disjunction of the above cases, which will be a finite formula, will define $\underline{d}$ such that $\phi(\underline{d}) = B(w)$, hence $\underline{d} = \rho_2\phi(w)$.  
    \item 
    Let us now consider $\rho_1\phi$, which corresponds to $T(w)$. This will be the row of bumped letters, which will mean that $c_i = 0$ for any $i\leq x$. For the later entries of $\underline{c}$, we will again have cases corresponding to the size of $m$.

    Suppose that as above, we have some $1\leq i\leq n-x$ such that \begin{align*}
        \sum_{j=1}^{i-1} r_{x+j}\leq m \leq \sum_{j=1}^{i} r_{x+j}
    \end{align*}
    Then we will bump all letters $x+1,\dots, x+i-1$, as well as some letters $x+i$. Therefore $c_{x+j} = r_{x+j}$ for $0<j<i$, and $c_{x+i} = m - \sum_{j=1}^{i-1} r_{x+j}$. Note that the length of $r_1$ is always exactly $m$ in this case. 
    
    Now suppose we are in the case
    \begin{align*}
        \sum_{j=1}^{n-x} r_{x+j}\leq m
    \end{align*}
    Then we will have that $c_{x+j} = r_{x+j}$ for each $0<j<n-x$. 
    
    As above, the disjunction of these cases yields a formula defining $\underline{c} = \rho_1\phi(w)$.\end{etaremune}\end{proof}
We will now show the definability of the stitch map
\begin{lemma}\label{stitchdefinable} 
The map $\phi^{-1}\Sigma(\phi(\ ),\phi(\ )): S^2\to S$ is definable.
\end{lemma}
\begin{proof}

 The condition for $\Sigma$ to have nontrivial action is definable via the following formula:
    
  Suppose $\underline{a}\ ,\ \underline{b} \in S$. Consider the set $B_1 =\{i~\in~\{1,\dots,k\}\ |\ c_i~=~x_m\dots x_1 \wedge x_1 = a\}$
    as in Lemma \ref{TBPEdefd}. Then $\phi(\underline{a})$ being a top tableau is definable by $\left(\bigwedge\limits_{i \in B_1} a_i = 0\right)$. Also, $\phi(\underline{b})$ being a row is definable by the formula $\left(\bigwedge\limits_{i\in \{1,\dots,k-n\}} b_i = 0\right).$

Now, let $e_a = \sum_{i \in B_a} a_i$. In order for it to be possible to stitch two inputs, we need $e_2 \leq b_{k-n+1}$, $e_3\leq b_{k-n+2} +b_{k-n+1}-e_2$, and so on. We can rearrange this to get the following compatibility condition: $$\underline{a}\in S \wedge \underline{b}\in S\wedge \left(\bigwedge\limits_{i \in B_1} a_i = 0\right)\wedge\bigwedge\limits_{i\in \{1,\dots,k-n\}} (b_i = 0)\wedge\bigwedge\limits_{i \in \{2,\dots,n\}}\left( \sum_{j = 2}^i e_i\leq \sum_{j = 1}^{i-1} b_{k-n+j}\right)$$

Where we take empty sums to be 0. Note that all sums used are finite, so we obtain a valid formula in Presburger arithmetic. Denote this compatibility formula $\gamma(\underline{a},\underline{b})$.

 When $\gamma$ is satisfied, we wish to construct $\underline{d} = \phi^{-1}\Sigma(\phi(\underline{a} ),\phi(\underline{b} ))$. In order to do this, we first briefly discuss what happens at each step during the multiplication algorithm in this case.

When a top tableau $t$ is multiplied by a compatible bottom row $r$, by the bumping property of Schensted's algorithm and the compatibility condition, each letter bumped by a letter of $r$ will bump the letter directly above it, which in turn bumps the letter directly above it, and so on until an entire column of $t$ has been bumped up by one space. As an example, consider $t = 34 223$ and $r = 112$. The multiplication algorithm will bump columns as follows:

\begin{ytableau}[*(Thistle)]
    3 &4 \\ 2&2&3
\end{ytableau}$\times \ytableaushort[*(SeaGreen)]{112} =\ $\begin{ytableau}[*(Thistle)]
    3\\ 2 &4 \\ *(SeaGreen)1&2&3
\end{ytableau}$\times  \ytableaushort[*(SeaGreen)]{12} =\ $\begin{ytableau}[*(Thistle)]
    3 &4 \\ 2&2\\ *(SeaGreen)1&*(SeaGreen)1& 3
\end{ytableau}$\times  \ytableaushort[*(SeaGreen)]{2} =\ $\begin{ytableau}[*(Thistle)]
    3 &4 \\ 2&2&3 \\ *(SeaGreen) 1 & *(SeaGreen) 1 & *(SeaGreen) 2
\end{ytableau}

In this way, the columns of $t$ are always bumped up in turn from left to right.
Due to this left-to-right bumping process, and the ordering of the coefficients $d_1,\dots, d_k$, running Schensted's algorithm will first calculate $d_1$, then $d_2$ and so on. Furthermore, as the algorithm runs, it will insert letters (which are themselves columns) of $r$ into columns of $t$. We will call this insertion ``using up" columns of $r$ and $t$. Suppose we use up $k$ columns $c_i^k$ of $t$ and $k$ columns $c_j^k$ of $r$. Then the corresponding coefficients $a_i$ and $b_j$ will need to be changed to $a_i-k$ and $b_j-k$ respectively. We can formalise this as a recursive process to compute $\underline{d}$:

Suppose we have calculated the coefficients $d_1,\dots,d_{i-1}$, and have obtained modified elements $\underline{a}^{i-1}$ and $\underline{b}^{i-1}$ in $S$ representing all columns that have not yet been used up in the stitch. We calculate the coefficient $d_i$ as follows:

$d_i$ is the coefficient of $c_i\in\mathcal{C}$, which we can write as $c_i = x_m\dots x_1 \in A^*$. Then $x_m\dots x_2$ and $x_1$ are two columns, which we denote respectively by $c_{i_T}$ and $c_{i_B}$ in $\mathcal{C}$. Call the coefficients corresponding to $c_{i_T}$ in $\underline{a}^{i-1}$ and $c_{i_B}$ in $\underline{b}^{i-1}$ by $a_{i_T}^{i-1}$ and $b_{i_B}^{i-1}$ respectively. By the structure of Schensted's algorithm it is straightforward to see that $d_i = \min(a_{i_T}^{i-1},b_{i_B}^{i-1})$, which is definable in Presburger arithmetic by 
\begin{align*}(a^{i-1}_{i_T}\leq b^{i-1}_{i_B}\land d_i = a^{i-1}_{i_T})\lor (b^{i-1}_{i_B}\leq a^{i-1}_{i_T}\land d_i = b^{i-1}_{i_B})\end{align*}

Now define $a^i_j = a^{i-1}_j$ for $j\neq i_T$, and $a^i_{i_T} = a^{i-1}_{i_T} - d_i$. Likewise, define $b^i_j = b^{i-1}_j$ for $j \neq i_B$ and $b^i_{i_B} = b^{i-1}_{i_B} - d_i$. This corresponds to the fact that these columns have now been used up in a stitch. Note that we will always get one of these coefficients being set to zero. This is clearly definable in Presburger arithmetic, and we will denote the formula for obtaining $\underline{a}^i$ and $\underline{b}^i$ from $\underline{a}^{i-1}$ and $\underline{b}^{i-1}$ by $\delta_i$. Note that when $i = 1$, we take $\underline{a}^0 = \underline{a}$ and $\underline{b}^0 = \underline{b}$, which allows us to calculate $d_1$ in terms of $\underline{a}^0$ and $\underline{b}^0$.

Now we get that $\phi^{-1}\Sigma(\phi(\ ),\phi(\ ))$ has graph $\{\underline{a},\underline{b},\underline{d}\}$ satisfying:
\begin{align*}
    \gamma(\underline{a},\underline{b})\land \exists \underline{a}^0\dots\exists\underline{a}^k\exists \underline{b}^0\dots\exists\underline{b}^k: (\underline{a}^0 = \underline{a})\land(\underline{b}^0 = \underline{b})\land \left(\bigwedge_{i \in \{1,\dots,k\}}d_i = \min(a^{i-1}_{i_T}, b^{i-1}_{i_B})\land \delta_i\right)
\end{align*}
\end{proof}

With the above lemmas in hand, we can now prove the following result.
\begin{prop}\label{multngood}
For any $x\in A$, the map $\mu_x$ is definable.
\end{prop}
\begin{proof} By the discussion before Lemma \ref{TBPEdefd}, $\mu_x$ is a composition of the maps $\phi^{-1}B\phi,\ \phi^{-1}T\phi,\ \rho_1\phi,\ \rho_2\phi,\ \phi^{-1}\Sigma(\phi(\ ),\phi(\ ))$, and multiplication of top tableaux. By Lemmas \ref{TBPEdefd}, \ref{rowmultn}, and \ref{stitchdefinable}, all five required maps are definable. 

To define $\underline{a}^5 = \phi^{-1}(T(\phi(\underline{a}))T(B(\phi(\underline{a}))x^m))$, first note that since $\underline{a}^5$ denotes a top tableau, we have that $a^5_i = 0$ for each $i\in B_1$ as defined in lemma \ref{TBPEdefd}. Furthermore, for each $i\notin B_1$, we have that $a^5_i$ is determined by the formula $\eta_{n-1}$ applied to $\phi^{-1}(T(\phi(\underline{a})))$ and $\phi^{-1}(T(B(\phi(\underline{a}))x^m))$. This determines $\eta_n$ by induction, with base case $\eta_2$ as detailed in section \ref{P2basecase}. This completes the proof. 
\end{proof}

Combining Proposition \ref{multngood} and Corollary \ref{simplification}, we obtain that $\phi^{-1}(\circ)$ is definable. Thus proving the following theorem

\begin{thrm}\label{thrmPninterpretable}
    The map $\phi:S \to P_n$ as defined above is an interpretation of $P_n$ in Presburger arithmetic. 
\end{thrm}
This reduces $FOTh(P_n)$  to the first order theory of Presburger arithmetic, which is decidable by lemma \ref{Presburger}, hence yielding the following result as a corollary.

\begin{thrm}
    For any $n\in\mathbb{N}$, the first order theory of $P_n$ decidable
\end{thrm}

\section{Definable submonoids and bi-interpretability}\label{sectionBiinterpret}

In a plactic monoid of any rank, the centre of $P_n$ will be generated by the column $c_1\in\mathcal{C}$ corresponding to the decreasing sequence $n(n-1)\dots 21\in A^*$. We have seen that the centre of any monoid is a definable subset, so we have $Z(P_n) = \{ c_1^n\ |\ n\in\mathbb{N}\}$ a definable subset of $P_n$ isomorphic to $(\mathbb{N},0,1,+,\leq)$, with $a\leq b$ in $\mathbb{N}$ corresponding to the formula $\exists y: c_1^ay=c_1^b$, and addition corresponding to monoid multiplication. 

We can therefore take for any $n$, the map $\psi:Z(P_n)\to \mathbb{N}$ to be an interpreting map of Presburger arithmetic in $P_n$. In this section, we will show that these mutually interpretable structures are, in fact, bi-interpretable. 

\subsection{The case for $P_2$}
In the style of section \ref{P2basecase}, take $\mathcal{C} = (t,1,2)$ with $Z(P_2) = \{t^n\ |\ n\in\mathbb{N}\}$. 
Firstly, note that $\phi: \mathbb{N}^3\to P_2$ as defined in that section  is also an interpreting map of $P_2$ into the $\mathbb{N}$ version of Presburger arithmetic, by rearranging any subtraction formulas $a = b-c$ to $a+c = b$. 

Let  $\psi:Z(P_2)\to \mathbb{N}$ be the interpretation of Presburger arithmetic in $P_2$ described above. Then taking $S$ as the definable subset $\{(n,0,0)\ |\ n\in\mathbb{N}\}\subset\mathbb{N}^3$, the map $\psi\phi: S\to \mathbb{N}$ is the isomorphism sending  $(n,0,0)$ to $n$, which is clearly definable. 

Now, consider $\phi\psi: Z(P_2)^3\to P_2$, sending $(t^a,t^b,t^c)$ to $w=t^a1^b2^c$. To show that this is definable, we will first show that the sets $N_1 = \{1^n|\ n\in\mathbb{N}\}$ and $N_2 = \{2^n|\ n\in\mathbb{N}\}$ are definable in $P_2$. 

First, consider the set elements $S_\ell$ that can be multiplied on the left to yield a central element $$
S_\ell = \{x\in P_2\ |\ \exists y: yx \in Z(P_2)\}
$$ and analogously, the set $S_r$ of elements hat can be multiplied on the right to yield a central element $$
S_r = \{x\in P_2\ |\ \exists y: xy \in Z(P_2)\}. 
$$
Since we know that 1 can be multiplied on the left to become central, and 2 can be multiplied on the right to become central, it is straightforward to see that $S_\ell = \{t^a1^b|\ a,b\in\mathbb{N}\}$ and $S_r = \{t^a2^b|\ a,b\in\mathbb{N}\}$. 

Notice that in $S_\ell$ (respectively $S_r$), each element is written as a product of a central element with an element of $N_1$ (respectively $N_2$). So if we insist that the central element in this product is always the identity, we will obtain the elements of $N_1$ (respectively $N_2$). In symbols, we can formalise this as the following conditions \begin{align*}
    x\in N_1 &\iff \left[x\in S_\ell \land \forall y \forall z: z\in Z(P_2) \land x = zy \implies z = \varepsilon\right]\\ 
    x\in N_2 &\iff \left[x\in S_r \land \forall y \forall z: z\in Z(P_2) \land x = zy \implies z = \varepsilon\right]
\end{align*}
so both sets $N_1$ and $N_2$ are definable. 

Now, given $t^b$ we can define $x = 1^b$ as the element of $N_1$ such that there is some $z\in N_2$ such that $zx = t^b$. Likewise, given $t^c$ we can define $y = 2^c$ as the element of $N_2$ such that there is some $z\in N_1$ such that $yz = t^c$. 

Then we take the image of $(t^a,t^b,t^c)\in Z(P_2)^3$ to be $\phi\psi(t^a,t^b,t^c) = t^axy$, with $x,\ y$ calculated as above. This is clearly definable, which completes the proof that this is in fact a bi-interpretation. 

\subsection{Submonoids generated by columns}

Notice that in the above section we showed that the submonoids $t^*, 1^*$, and $2^*$ generated by each of the columns are definable in $P_2$.  We will now show that this is a general fact in any plactic monoid, which will be useful for constructing a bi-interpretation.
\begin{thrm}\label{columnsubmonoids}
    For any plactic monoid $P_n$ with column generating set $\mathcal{C}$, for each $c_i\in\mathcal{C}$, the submonoid $c_i^* = \{c_i^n\ |\ n\in\mathbb{N}\}$ is definable in $P_n$.
\end{thrm}
\begin{proof}
    We will proceed by induction, with base case $P_2$ from the previous section. 

    As before, consider the set $S_\ell$ of all elements that can be multiplied on the left to yield a central element \begin{align*}
S_\ell &= \{x\in P_n\ |\ \exists y: yx \in Z(P_n)\}.
\end{align*}
Note that the centre of any monoid is definable, so $S_\ell$ is a definable set. We claim that any element of $S_\ell$ will have all its columns in the form
\ytableausetup{centertableaux}
\begin{align*}
    u_i = \begin{ytableau}
        i\\ \scriptstyle i-1\\ \vdots \\ 1
    \end{ytableau}
\end{align*}
\ytableausetup{aligntableaux = bottom}

i.e. a general $X\in S_\ell$ will be of the form $u_n^{x_n}\dots u_1^{x_1}$. Indeed, for each such $X$ we have a $Y = f_{n}^{x_{n-1}}f_{n-1}^{x_{n-2}}\dots f_2^{x_1}$, where 
\ytableausetup{centertableaux}
\begin{align*}
    f_i = \begin{ytableau}
        n\\ \scriptstyle n-1\\ \vdots \\ i
    \end{ytableau}
\end{align*}
\ytableausetup{aligntableaux = bottom}

and $YX$ is central, so $X$ is a member of $S_\ell$. 

To show that there are no other elements in $S_\ell$,  consider the structure of Schensted's algorithm. Given two elements $u, v \in P_n$, for each $x\in A$ that appears in the row reading of $v$, the corresponding $x$ in the row reading of $uv$ will either be on the same row as it started in $v$, or on a lower row. As an example, in $P_3$, consider the product of $u = 3\ 223\ 1122$ and $v = 3\ 22\ 1123$, here written in row reading. Then after running Schensted's algorithm with colour coding, we can see that

$$\ytableaushort[*(SeaGreen)]{3,223,1122} \times \ytableaushort[*(Thistle)]{3,22,1123} = \begin{ytableau}
    *(SeaGreen)3&*(SeaGreen)3&*(Thistle)3\\ *(SeaGreen)2&*(SeaGreen)2&*(SeaGreen)2&*(SeaGreen)2\\ *(SeaGreen)1&*(SeaGreen)1&*(Thistle)1&*(Thistle)1&*(Thistle)2&*(Thistle)2&*(Thistle)2&*(Thistle)3
\end{ytableau}. $$
This will be true in general because when you run Schensted's algorithm on a string in $A^*$, a letter may only be bumped by other letters further to the right in the string. So any $x$ that started in $v$ will only be moved by other letters of $v$, even in the product $uv$. But since each letter must bump the leftmost letter in the row, it might happen that $x$ will not be bumped as many times in the product as it was in $v$, in which case it might end up on a lower row. Most importantly, any $x$ that started on a given row of $v$ cannot be in a higher row of $uv$ than the row it started in.

Note that in a central element (a power of $c_1$), any $x\in A$ must appear in the row that is $x$-th from the bottom\footnote{i.e. all 1's appear on the bottom row, all 2's on the second from bottom row, and so on.}. The columns of the form $u_i$ are all the columns in $\mathcal{C}$ where each $x$ appears only on the $x$-th row from the bottom. 

Now, suppose $w\in P_n$ is a tableau with at least one of its columns not in the form of some $u_i$. Then there will be some letter $x$ in the tableau of $w$ which is lower than the $x$-th row from the bottom. But because multiplication on the left cannot move letters of $w$ to a higher row, for any $Y$ we must have an $x$ in $Yw$ in a row lower than $x$-th from the bottom. Hence $Yw$ is  never central, so $w$ is not a member of $S_\ell$. 

Next, notice that any $X\in S_\ell$ can be written as the product of $\Tilde{X} = u_n^{x_n}\dots u_2^{x_2}$ and $u_1^{x_1}$. This is useful, because $\Tilde{X}$ commutes with the generator 2, while $u_1^{x_1}$ does not. Furthermore, the column $u_1$ corresponds to the generator 1. So, similar to the case of $P_2$, we may define the submonoid $1^*$ as the set satisfying the formula $$x \in S_\ell\land \forall y \forall z: z2 = 2z \land x = zy \implies z = \varepsilon.$$
Since $\Tilde{X}$ commutes with 2, the formula above will insist that $\Tilde{X} = \varepsilon$, leaving only the elements of $S_\ell$ of the form $u_1^x$. Therefore, this is precisely the set $1^* = \{1^n\ |\ n\in\mathbb{N}\}$, as in the $P_2$ case. 

Using this, we can define the set satisfying the formula 
$$\forall x \forall y \forall z: w = xyz \land (y \in 1^*) \implies y = \varepsilon.$$
This formula disallows any instance of the generator 1, so defines the subset of $P_n$ generated by $\{2,\dots,n\}$. This is a submonoid isomorphic to $P_{n-1}$. 

Since it is definable in $P_n$, any set definable in this submonoid is also definable in $P_n$. By our induction hypothesis, any column generated submonoid of $P_{n-1}$ is definable in $P_{n-1}$. Therefore, by this hypothesis we get that $x^*$ is definable in $P_n$ for any $x\in A = \{1,\dots,n\}$. 

Now, we know that each column $c_i\in\mathcal{C}$ will correspond to a nonempty subset of $A$. Of these, every column except $c_1$ will have at least one letter $x\in A$ omitted. But since $x^*$ is definable in $P_n$ for each $x$, we can define a submonoid isomorphic to $P_{n-1}$ generated by $A\setminus\{x\}$, using a similar formula to the one above: $$\forall x \forall y \forall z: w = xyz \land y \in x^* \implies y = \varepsilon.$$

Each $c_i\in\mathcal{C}$ except $c_1$ will be an element of at least one of these definable submonoids. Then, using our induction hypothesis, it follows that $c_i^*$ is definable in $P_n$ for each $c_i\in\mathcal{C}$ except $c_1$. But $c_1^*$ is the centre of $P_n$, which is also definable, thus completing the proof.\end{proof}

\subsection{A bi-interpretation for plactic monoids}

\begin{thrm}\label{thrm:bi-interpretPnPres}
    $P_n$ and Presburger Arithmetic are bi-interpretable for each $n\in\mathbb{N}$.
\end{thrm}

\begin{proof}
We start with the map $\phi:S\to P_n$ from section \ref{Pngeneralcase}, which by Theorem \ref{thrmPninterpretable} is an interpretation, and the interpretation $\psi:Z(P_n)\to\mathbb{N}$. 

Taking $T\subset S$ corresponding to the preimage of $Z(P_n)$, we have that $$T = \{(x,0,\dots,0)\ |\ x\in\mathbb{N}\}.$$ Thus $\psi\phi:T\to\mathbb{N}$ is the obvious bijection sending $(n,\dots,0)$ to $n$, and clearly an isomorphism in the language of Presburger arithmetic. 

On the other hand, considering the reverse composition we have $\phi\psi: V\to P_n$, where $V\subset Z(P_n)^k$ is the subset defined through the incompatibility conditions on columns, which simply insist that certain entries be identity. This map $\phi\psi$ will identify a tuple $\underline{a} = (c_1^{a_1},c_1^{a_2},\dots,c_1^{a_k})$ with an element $c_1^{a_1}c_2^{a_2}\dots c_k^{a_k}\in P_n$. 

By the argument in lemma \ref{sequencemultn}, this is equivalent to identifying $\underline{a}$ with an element $\alpha_1^{a_{\beta_1}}\alpha_2^{a_{\beta_2}}\dots \alpha_t^{a_{\beta_t}}$. Since the sequences $\alpha$ and $\beta$ are fixed, this means that as long as we can identify $c_1^{a}$ with $x^a$ for each $x\in A$, we can define the element  $\alpha_1^{a_{\beta_1}}\alpha_2^{a_{\beta_2}}\dots \alpha_t^{a_{\beta_t}}$ in terms of $\underline{a}$. Namely, this will show that the map $\phi\psi$ is definable in the language of $P_n$, which will complete the proof that we have built a bi-interpretation.

Recall the columns of the form $u_i$ and $f_i$ from the proof of theorem \ref{columnsubmonoids}. Note that for each $x\in A$ we have that $$f_{x+1}xu_{x-1} = c_1.$$

Furthermore, by theorem \ref{columnsubmonoids}, we know that $x^*,\ u_{x-1}^*,$ and $f_{x+1}^*$ are definable subsets of $P_n$. So, if we consider $w\in x^*,\ y\in f_{x+1}^*$, and $z\in u_{x-1}^*,$ we will have that $ywz = c_1^a$ precisely when $w = x^a$. So, the formula $$w \in x^* \land \exists y \exists z: y\in f_{x+1}^*\land z\in u_{x-1}^* \land ywz = c_1^a$$ defines $x^a$ in terms of $c_1^a$. Hence the image $\alpha_1^{a_{\beta_1}}\alpha_2^{a_{\beta_2}}\dots \alpha_t^{a_{\beta_t}}$ of $\underline{a}$ under $\phi\psi$ is definable in the language of $P_n$.
\end{proof}

Due to the transitivity of interpretations, we also have the following nice corollary.
\begin{cor}
    For any $m,n\in\mathbb{N}$, $P_n$ and $P_m$ are bi-interpretable.
\end{cor}

\section{The Diophantine problem in the infinite case}\label{sectionInfinite}

We note that the above interpretations were constructed algorithmically in a uniform way. That is to say, there will exist an effective procedure which, given $n$, will construct the interpreting map for $P_n$. The procedure runs as follows:
\begin{enumerate}
\item Generate the interpretation for $P_{n-1}$
    \item Given $n$, generate the power set of $\{1,\dots,n\}$ except the empty set.
    \item Enumerate this set by the order $\sqsubseteq$ on columns. Since each column is a decreasing sequence of elements in $\{1,\dots,n\}$, each column corresponds to a unique element of the power set. 
    \item Run Schensted's algorithm on each pair of columns. If the output of running Schensted's algorithm on $c_ic_j$ is not $c_ic_j$, then $(i,j)$ is an incompatible pair. 
    \item Generate the formula defining $S$ by conjuncting with $x_i = 0 \lor x_j = 0$ for each incompatible pair discovered in step 3.
    \item Generate the formula defining equality in terms of the formula defining $S$.
    \item Generate the formula for $\mu_x$ in terms of the interpretation for $P_{n-1}$
    \item Generate the sequences $\alpha$ and $\beta$ from lemma \ref{sequencemultn}. Steps 7 and 8 yield a formula defining multiplication. 
\end{enumerate}
Step 1 will repeat recursively until we reach $P_2$, which can be written explicitly as in section \ref{P2basecase}.

\subsection{The plactic monoid of all tableaux}
We consider $A = \mathbb{N}\setminus\{0\}$ with $K_\mathbb{N}$ the set of Knuth relations for all triples $(x,y,z)~\in~\mathbb{N}^3$. Then the associated plactic monoid $P(\mathbb{N})$ is the monoid of \emph{all} semistandard Young tableaux. Despite the work in this paper, the question of deciding the theory of $P(\mathbb{N})$ remains open. However, we present an algorithm, by uniformity, for deciding the Diophantine problem for $P(\mathbb{N})$. 

\begin{lemma}\label{PNhomlemma}
    For any $n\in\mathbb{N}$ the map $\phi:P(\mathbb{N})\to P_n$, defined on generators as $\phi(k) = k$ if $k\leq n$ and $\phi(k) = \varepsilon$ if $k>n$ and extended to words in the natural way, is a homomorphism.
\end{lemma}
\begin{proof}
    Considered as a map from $\mathbb{N}^*$ to $\{1,\dots,n\}^*$, $\phi$ is clearly a homomorphism. It only remains to show that $\phi$ is well defined as a map from $P(\mathbb{N})$ to $P_n$. We will show this by proving that each Knuth relation in $K_\mathbb{N}$ will map to a relation that holds in $P_n$. 

    Suppose $u = xzy$ and $v = zxy$, for $x\leq y < z$. If $z\leq n$ then $\phi(u) = u,\ \phi(v) = v,$ and $u=v$ in $P_n$ so there is nothing to prove. If $z>n$ then $\phi(u)=\phi(x)\phi(y)=\phi(v)$, as $\phi(z) = \varepsilon$. Thus $\phi(u) = \phi(v)$ will always hold in $P_n$. An analogous argument shows $\phi(u) = \phi(v)$ for $u = yxz$ and $v=yzx$ with $x<y\leq z$. 
\end{proof}
\begin{thrm}
    The Diophantine problem for $P(\mathbb{N})$ is decidable.
\end{thrm}
\begin{proof}
Suppose we are given some equation $u_1X_1\dots X_nu_{n+1} = v_1Y_1\dots  Y_mv_{m+1}$. Denote this equation by $\varphi$. Define the \emph{support} of an element $u\in P(\mathbb{N})$ to be the set of all numbers appearing in $u$. Further, define the support of $\varphi$ to be all letters appearing in the supports of $u_i$ and $v_i$
$$supp(\varphi) = \bigcup_{i\leq n+1} supp(u_i)\cup\bigcup_{j\leq m+1} supp (v_j)$$
Let $k = \max(supp(\varphi))$. Then by the above proposition there exists a homomorphism $\phi:P(\mathbb{N})\to P_k$. Since each $u_i$ and $v_j$ is an element of $P_k$, we get that $\phi(u_i) = u_i$ and $\phi(v_j) = v_j$. 

Now, suppose $\varphi$ has a solution $(x_1,\dots,x_n,y_1,\dots,y_m)\in P(\mathbb{N})^{m+n}$. 

Then $(\phi(x_1),\dots \phi(x_n),\phi(y_1),\dots,\phi(y_m))\in P_k^{n+m}$ will also be a solution to $\varphi$, since $\phi$ is a homomorphism. Thus $\varphi$ has a solution in $P(\mathbb{N})$ if and only if it has a solution in $P_k$. 

Now, since there is a uniform algorithm for deciding first order sentences in $P_k$ for any $k$, we obtain the following procedure for solving Diophantine problems in $P(\mathbb{N})$:

\begin{enumerate}
    \item Given $\varphi$ as input, calculate $k = \max(supp(\varphi))$.
    \item Generate the interpretation of $P_k$ into Presburger arithmetic
    \item Interpret the sentence $$\exists X_1\dots \exists X_n\exists Y_1\dots\exists Y_m: u_1X_1\dots X_nu_{n+1} = v_1Y_1\dots  Y_mv_{m+1}$$
    in Presburger arithmetic using the interpretation of $P_k$, and check whether it holds.
\end{enumerate}\end{proof}

\subsection{A plactic monoid on integers}
We need not restrict ourselves to plactic monoids generated by $\mathbb{N}$.

Let's consider instead tableaux with labels taken from $\mathbb{Z}$. By the total order on $\mathbb{Z}$ we obtain a set $K_\mathbb{Z}$ of Knuth relations on triples $(x,y,z)\in\mathbb{Z}^3$. Define the plactic monoid on integers to be $P(\mathbb{Z}) = \langle\mathbb{Z}\ |\ K_\mathbb{Z}\rangle$. This is an infintely generated plactic monoid, but note that $P(\mathbb{N})$ and $P(\mathbb{Z})$ are not isomorphic.

Indeed, suppose $\psi: P(\mathbb{Z})\to P(\mathbb{N})$ were an isomorphism. Then for some $y\in P(\mathbb{Z})$ we have $\psi(y) = 1$. Since 1 is irreducible, we must have $y\in\mathbb{Z}$. Consider $x<y<z,\ x,y,z\in\mathbb{Z}$. Then by irreducibility, $\psi(x),\psi(z)\in\mathbb{N}$. Thus we have some $a,b\in\mathbb{N}$ such that $1ab = 1ba$. Such an equality cannot hold in $P(\mathbb{N})$.

Given $\varphi$ a Diophantine equation in $P(\mathbb{Z})$, we will have $supp(\varphi)$ a finite totally ordered set. This set will have some smallest element $a\in\mathbb{Z}$ and some largest element $b\in\mathbb{Z}$. Then the interval $[a,b]\subset\mathbb{Z}$ has size $k = b-a+1$, and  we can define an order preserving injective map from $supp(\varphi)$ to $\{1,\dots,k\}$. We will extend this map to a homomorphism.

\begin{lemma}
    Let $\{z_1<z_2<\dots < z_n\}$ be a finite set of integers with their standard order. Then the map $\phi:P(\mathbb{Z})\to P_k$, with $k = z_n-z_1+1$ defined on generators by $$\phi(z) = \begin{cases}
        \varepsilon,\ z<z_1\\
       z-z_1+1,\ z \in [z_1,z_n]\\
     \varepsilon,\  z> z_n
    \end{cases}$$
    and extended to words in the natural way, is a homomorphism
\end{lemma}
    \begin{proof}
As in lemma \ref{PNhomlemma}, consider $u = xzy$ and $v = zxy$, for $x\leq y z$. If more than one letter in $\{x,y,z\}$ is mapped to $\varepsilon$, there is nothing to prove. Likewise if no letters are mapped to $\varepsilon$. If only one letter is mapped to $\varepsilon$, then this is either $x$, yielding $\phi(u) = \phi(z)\phi(y) = \phi(v)$, or this letter is $z$, yielding $\phi(u) = \phi(x)\phi(y) = \phi(v)$. An analogous argument holds for all other Knuth relations.
   \end{proof}

Thus, as in the case above, any Diophantine equation $\varphi$ is solvable in $P(\mathbb{Z})$ if and only if it has a solution in a fixed finite rank plactic monoid. Therefore, by uniformity of the above algorithm, we obtain the following corollary
\begin{cor}
    The Diophantine problem for $P(\mathbb{Z})$ is decidable.
\end{cor}

\subsection{Two open questions}
\begin{enumerate}
    \item Is the first order theory of $P(\mathbb{N})$ decidable? It is known that this monoid satisfies no identities \cite{johnson2021tropical}, and the above proof shows it has decidable Diophantine problem. Can this be extended to the whole theory? What about in the $P(\mathbb{Z})$ case?
    \item Do infinite rank plactic monoids defined on other generating sets have decidable Diophantine problem? For example, does $P(\mathbb{Q}) = \langle\mathbb{Q}\ |\ K_\mathbb{Q}\rangle$ have decidable Diophantine problem? What about $P(L)$ for an arbitrary recursive total order? More generally, do such monoids have decidable theory as well?
\end{enumerate}
\section*{Acknowledgements}
This research was conducted during my master's study at the University of East Anglia, and will form part of my thesis. I thank Robert Gray for his support and feedback as supervisor, and Lorna Gregory for her feedback and useful discussion of the model theoretic background. I also thank the PhD student community at the UEA for their support. 

\newpage
\printbibliography
\end{document}